\documentclass[11pt]{article}
\RequirePackage{fix-cm}
\usepackage{graphicx}
\usepackage{amsmath}
\usepackage{amsthm}
\usepackage{amsfonts}
\usepackage{amssymb}
\usepackage{graphicx}
\usepackage{color}
\usepackage{rotating}
\usepackage{authblk}
\usepackage{url}
\usepackage{tcolorbox}
\usepackage{tikz,pgfplots} 
\usetikzlibrary{arrows,shadings,shadows,automata,shapes,calc,positioning,patterns,decorations.markings,decorations.pathmorphing,snakes,plotmarks}
\usepackage{soul}
\usepackage{footnote}
\usepackage{geometry}
\usepackage{hyperref}
\usepackage{authblk}
\usepackage{algorithm}
\usepackage{algpseudocode}

\newcommand{\Z}{\mathbb{Z}}

\newcommand{\E}{\mathbb{E}}

\newcommand{\e}{\varepsilon}

\DeclareMathOperator{\OPT}{\textsc{OPT}_{\ell_0}}

\newcommand{\ones}{\boldsymbol{1}}
\newcommand{\zeros}{\boldsymbol{0}}

\newtheorem{theorem}{Theorem}
\newtheorem{proposition}{Proposition}
\newtheorem{lemma}{Lemma}

\newtheorem{corollary}{Corollary}
\newtheorem{observation}{Observation}

\makeatletter
\renewcommand\footnoterule{%
  \kern-3\p@
  \hrule\@width \textwidth
  \kern2.6\p@}
\makeatother

\makeatletter
\renewcommand*{\@fnsymbol}[1]{\ensuremath{\ifcase#1\or *\or \dagger\or \ddagger\or **\or \mathsection\or \mathparagraph\or \|\or  \dagger\dagger
   \or \ddagger\ddagger \else\@ctrerr\fi}}
\makeatother

\newcounter{mynotes}
\newcommand{\remove}[1]{}

\newcommand{\optz}{\textsc{OPT}_{\ell_0}}
\newcommand{\opto}{\textsc{OPT}_{\ell_1}}

\begin{document}

\title{Sparse principal component analysis and its $\ell_1$-relaxation}

\author[s]{Santanu S. Dey\thanks{santanu.dey@isye.gatech.edu}}
\author[p]{Rahul Mazumder\thanks{rahulmaz@mit.edu}}
\author[c]{Marco Molinaro\thanks{mmolinaro@inf.puc-rio.br}}
\author[a]{Guanyi Wang\thanks{gwang93@gatech.edu}} 
\affil[a,s]{\small School of Industrial and Systems Engineering, Georgia Institute of Technology} 
\affil[p]{\small Operations Research Center, Massachusetts Institute of Technology}
\affil[c]{\small Computer Science Department, Pontifical Catholic University of Rio de Janeiro}
\maketitle

\begin{abstract}
Principal component analysis (PCA) is one of the most widely used dimensionality reduction methods in scientific data analysis. In many applications, for additional interpretability, it is desirable for the factor loadings to be sparse, that is, we solve PCA with an additional cardinality ($\ell_{0}$-norm) constraint. The resulting optimization problem is called the sparse principal component analysis (SPCA). One popular approach to achieve sparsity is to replace the $\ell_{0}$-norm constraint by an $\ell_{1}$-norm constraint. In this paper, we prove that, independent of the data, the optimal objective function value of the problem with $\ell_0$ constraint is within a constant factor of the the optimal objective function value of the problem with $\ell_1$ constraint. To the best of our knowledge, this is the first formal relationship established between the $\ell_0$ and the $\ell_1$ constraint version of the problem.
\\
\smallskip
\noindent \textbf{Keywords.} $\ell_1$ regularization, Sparsity, Principal component analysis 

\end{abstract}

\section{Introduction}
\label{sec:intro}
\paragraph{Principal component analysis (PCA).} PCA~\cite{hotelling1933analysis} is one of the most widely used dimensionality reduction methods pervasive in statistics, data science and scientific data analysis~\cite{jolliffe2002principal}. Given a data matrix ${Y}_{m \times n}$ (with $m$ samples and $n$ features; and each feature is centered to have zero mean), the task of PCA is to find a direction ${x} \in \mathbb{R}^{n}$ (with $\|{x} \|_{2} = 1$) such that it maximizes the variance of a weighted combination of the features, given by: ${Y} {v}$. If $A:=\frac{1}{m} Y^{\top}Y$ denotes the sample covariance matrix of ${Y}$, then a principal component (PC) direction can be found by 
\begin{equation}\label{pca-var-1}
\textup{max}_x ~~~ {x}^{\top}A {x} ~~ \text{s.t.} ~~  \|{x} \|_{2}=1.
\end{equation}
A maximizer $\hat{{x}}$ of (\ref{pca-var-1}) can be computed in polynomial time via a rank one eigendecompostion~\cite{GVL83} of $A$. The entries of $\hat{x}$ are known as 
the factor loadings, and they lead to 
the first principal component direction ${Y}{\hat{x}}$, a linear combination of the features with maximal variance. 
PCA is widely used in microarray analysis~\cite{hastie2000gene,misra2002interactive}, handwritten zip code classification~\cite{FHT-09-new}, human face recognition~\cite{hancock1996face},
image processing~\cite{jenatton2010structured}, text processing~\cite{uuguz2011two}, financial analysis~\cite{paul2012augmented,zhang2012sparse} among others~\cite{naikal2011informative}.

\paragraph{Sparse PCA.}
An obvious drawback of PCA is that all the entries of $\hat{{x}}$ are nonzero, which leads to the PC direction being a linear combination of all features -- this impedes interpretability~\cite{cadima1995loading,jolliffe2003modified,zou2006sparse}. In microarray analysis for example, when ${Y}$ corresponds to the gene-expression measurements for different samples, it is desirable to obtain a PC direction which involves only a handful of the features (i.e., genes) for interpretation purposes. In financial applications (where, $A$ denotes the sample covariance matrix of stock-returns), a sparse subset of stocks that are responsible for driving the first PC direction may be desirable for interpretation purposes. Thus in many scientific and industrial applications, for additional interpretability, it is desirable for the factor loadings to be sparse, i.e., few of the entries in $\hat{x}$ are nonzero and the rest are zero. This motivates the notion of a sparse principal component analysis (SPCA)~\cite{jolliffe2003modified,hastie2015statistical}, wherein, in addition to maximizing the variance, one also desires the direction of the first PC to be sparse in the factor loadings. The most natural optimization formulation of this problem, modifies criterion~\eqref{pca-var-1} with an additional sparsity constraint on $x$ leading to:
\begin{equation}\label{pca-var-l0}
\textup{max}_x~~~  {x}^{\top}A {x} ~~ \text{s.t.} ~~  \|{x} \|_{2}=1, \| x \|_0 \leq k,
\end{equation}
where, $ \| x \|_0 \leq k$ allows at most $k$ of the entries in $x$ to be nonzero.

In addition to interpretability, sparsity is a key dimensionality reduction tool needed for meaningful statistical inference. For example, suppose $Y$ is a data matrix 
that is generated from a spiked covariance model with $\Sigma =  \tau \theta \theta^{\top} + \sigma^2 \mathbb{I}$ where, $\theta \in \mathbb{R}^{n}$ with $\| \theta \|_{2}=1$ and $\mathbb{I}$ denotes the identity matrix.  Under the classical asymptotic regime, i.e., as the number of samples $m \rightarrow \infty$ with $n$ fixed, the first PC direction or the eigenvector of the sample covariance matrix $A$ is consistent~\cite{anderson2003} (up to sign changes) for the population version $\theta$. However, when $m,n$ are comparable with $\tfrac{m}{n} \rightarrow c \in (0,\infty)$ as $m \rightarrow \infty$ this classical consistency theory breaks down. The sample PC may no longer be consistent for the population version $\theta$, if $\tau/\sigma^2$ is sufficiently small -- see~\cite{johnstone2009consistency} for additional details. 
In such situations, additional structure such as sparsity assumptions on $\theta$ are called for. 

The SPCA problem has received significant attention in the wider statistics community since 1990s~\cite{cadima1995loading}; and influential follow-up work by~\cite{jolliffe2003modified,zou2006sparse,shen2008sparse,wht_09,johnstone2009consistency}, among many others. \cite{journee2010generalized,luss2013conditional} study well-grounded nonlinear optimization algorithms based on modifications of the power method for SPCA-type problems. 

\paragraph{Enforcing $\ell_1$ constraint in place of $\ell_0$ constraint.}
Unlike usual PCA, the sparse variant, Problem~\eqref{pca-var-l0} is no longer easy to compute---several approaches and computational schemes have been proposed to address this problem. One of the most popular approaches is to relax the cardinality constraint $\| v \|_0 \leq k$ by an $\ell_{1}$ aka Lasso~\cite{Ti96} constraint, leading to
\begin{equation}\label{pca-var-l0-l1}
\textup{max}_x~~~  {x}^{\top}A {x} ~~ \text{s.t.} ~~  \|{x} \|_{2}=1, \| x\|_1 \leq \delta,
\end{equation}
for some $\delta>0$. Criterion~\eqref{pca-var-l0-l1} was proposed in~\cite{jolliffe2003modified}. Criterion~\eqref{pca-var-l0-l1} is appealing as it uses a soft version of sparsity akin to Lasso regression: the $\ell_{1}$-constraint on $x$ induces both sparsity and shrinkage in a continuous fashion via the tuning parameter $\delta$; unlike Problem~\eqref{pca-var-l0} which produces a discrete set of solutions for every $k \in [n]$. In addition, the $\ell_{1}$-constraint may be suitable when some entries of $x$ are small (instead of being exactly zero) and the others are large. The papers~\cite{vu2012minimax,birnbaum2013minimax} have studied minimax optimal properties of the estimator~\eqref{pca-var-l0-l1}
under a spiked covariance model, under the assumption that the population eigenvector lies in the $\ell_{1}$ ball.  

Problem~\eqref{pca-var-l0-l1} is a continuous optimization problem unlike Problem~\eqref{pca-var-l0} and hence more amenable to techniques in 
nonlinear continuous optimization: \cite{jolliffe2003modified} propose to use a projected gradient method for Problem~\eqref{pca-var-l0-l1}.
Note however that unlike the Lasso version of best-subset selection\footnote{Best subset selection refers to the task of best explaining a response $r\in \mathbb{R}^{m}$ as a linear combination of $k$ features: 
$\min \{\| r - F \beta \|_{2}^2: \|\beta \|_0 \leq k\}$, where, $F_{m \times n}$ is the data-matrix with $m$ samples and $n$ features.}
which is convex; Problem~\eqref{pca-var-l0-l1} is a difficult nonconvex optimization task; and computing optimal solutions may be difficult.
\cite{wht_09} (see also Chapter~8~\cite{hastie2015statistical}) argue that developing an iterative scheme towards optimization of~\eqref{pca-var-l0-l1} is 
not straightforward and hence consider a close cousin given by:
\begin{equation}\label{pca-var-l0-l1-am}
\textup{max}_{x,y}~~~  {y}^{\top} Y {x} ~~ \text{s.t.} ~~  \|{y} \|_{2}=1, \| x \|_1 \leq \delta, \| x \|_{2} = 1,
\end{equation}
 where, $Y$ is the data-matrix (recall that $A = \tfrac{1}{m} Y^{\top}Y$). \cite{wht_09,hastie2015statistical} propose a clever
 alternating optimization scheme for Problem~\eqref{pca-var-l0-l1-am}.

\paragraph{Our result: formal relation between enforcing $\ell_1$ constraint and the $\ell_0$ constraint.}
Unlike the literature on sparse regression, the literature on SPCA treats the $\ell_0$ and $\ell_1$ constraints separately, for example, deriving separate semi-definite programming (SDP) relaxations~\cite{jordan_07,zhang2012sparse}. To the best of our knowledge, there is no theoretical results comparing the solutions or the optimal objective function value of the problems with $\ell_0$ and $\ell_1$ constraints.

In the context of SPCA, note that the constraints $\|x\|_0 \leq k$ and $\|x\|_2 \leq 1$ together imply that $\|x\|_1 \leq \sqrt{k}$. Thus, for $\delta = \sqrt{k}$, (\ref{pca-var-l0-l1}) is relaxation of (\ref{pca-var-l0}). It therefore makes sense to compare (\ref{pca-var-l0}) and (\ref{pca-var-l0-l1}) with $\delta = \sqrt{k}$. Henceforth we refer to (\ref{pca-var-l0-l1}) with $\delta = \sqrt{k}$ as the $\ell_1$-relaxation of SPCA.

\emph{In this paper we prove that, independent of $A$, the optimal objective function of SPCA (i.e, (\ref{pca-var-l0})) is within a constant factor of the optimal objective function of the $\ell_1$-relaxation of SPCA (i.e. (\ref{pca-var-l0-l1}) with $\delta = \sqrt{k}$).} Our proof of this result is via a randomized rounding argument, thus yielding a constant factor approximation algorithm to solve SPCA assuming we have access to the optimal solution of its $\ell_1$-relaxation. Moreover, our result holds more generally when ${x}^{\top}A {x}$ in the objective is replaced by any semi-norm. Therefore, instead of maximizing $\|Yx\|^2_2$ (which is the same as maximizing ${x}^{\top}A {x}$), if we maximize $\|Yx\|_1$ in (\ref{pca-var-l0}) and (\ref{pca-var-l0-l1}) with $\delta = \sqrt{k}$, the constant factor result still holds. We note that such $\ell_1$-norm objectives in the context of PCA has been studied~\cite{mccoy2011two}.     

It is intriguing to compare our result on the role played by $\ell_{1}$-constraint in the context of PCA to the same in the context of best-subsets selection. The pioneering work by Donoho~\cite{donoho2006}, Candes and Tao~\cite{candes2005decoding}, and Candes et al.~\cite{candes2006stable}, showed that sparse solutions to under-determined system of equations may be retrieved by replacing the $\ell_0$-pseudo norm by a $\ell_1$ norm. However this result holds only under the assumption that the data matrix satisfies certain conditions such as the ``restricted isometry property". The noisy version of the problem requires additional assumptions on the problem data, and for support recovery additional assumptions (such as the irrepresentable condition) are needed--see for e.g.,~\cite{ZY2006,buhlmann2011statistics}. Our result on the constant factor approximation; on the other hand, does not require any assumption on $A$ -- and holds universally -- making it quite different from the existing results for $\ell_{0}$-$\ell_{1}$-equivalence in the context of sparse regression. We do note however, that the $\ell_{1}$-version of the problem for sparse linear regression is a convex optimization problem; and hence computable in polynomial time -- both the problems~(\ref{pca-var-l0}) and~(\ref{pca-var-l0-l1}) are NP-hard.  

We finally note here that the paper~\cite{FountoulakisKKD17} presented for the first time the simple randomized algorithm used for our analysis. This algorithm starts with a solution of $\ell_1$-relaxation of SPCA (i.e. (\ref{pca-var-l0-l1}) with $\delta = \sqrt{k}$) and randomly rounds it to produce sparsity. Loosely speaking, the result obtained in~\cite{FountoulakisKKD17} is of the following form: While with high probability the additive difference in the objective function value of $\ell_1$-relaxation and the objective function value of the randomly obtained vector is bounded by $\epsilon$, the expected sparsity of the randomly obtained vector is $\frac{200 k}{\epsilon}$ which is significantly larger than $k$. Therefore, this result does not establish a relationship between SPCA and the $\ell_1$-relaxation for the same value of $k$. Our analysis explicitly accounts for the positive semi-definiteness of $A$, which is not used in the analysis presented in~\cite{FountoulakisKKD17}.


\section{Main results}
\label{sec:main}
For an integer $t \geq 1$, we use $[t]$ to describe the set $\{1, \dots, t\}$. Also, we represent the $j^{\text{th}}$ unit vector, the vector of ones, and the vector of zeros in appropriate dimension by $e_j$, $\ones$, and $\zeros$, respectively. 

Since the square root function is monotonic, note that the objective function in (\ref{pca-var-l0}) and (\ref{pca-var-l0-l1}) can be replaced by $\sqrt{{x}^{\top}A {x}}$ and the resulting problem has the same set of optimal solutions. We denote $\sqrt{{x}^{\top}A {x}}$ by $\|x\|_{A}$.

As mentioned in the previous section, our main result holds for more general objective functions than that of $\|x\|_{A}$. Let $\phi:\mathbb{R}^n\rightarrow \mathbb{R}_{+}$ be a \emph{semi-norm}, i.e., (i) $\phi$ is positively-homogenous: $\phi(\lambda x) = \lambda \phi(x)$ for all $\lambda \geq 0$, (ii) $\phi$ is subadditive: $\phi(u  + v) \leq \phi(u) + \phi(v)$ for all $u, v\in \mathbb{R}^n$, (iii) $\phi$ is nonnegative: $\phi(u) \geq 0$ for all $u \in \mathbb{R}^n$, and (iv)  $\phi(\textbf{0}) = 0$. Conditions (i) and (ii), imply that $\phi$ is a convex function. Also note that $\phi(x) = 0$ does not imply that $x = \textbf{0}$. 

Since $A$ is positive semi-definite, it is straightforward to verify that $\|x\|_{A}$ is semi-norm. We now present the general version of sparse PCA, which we call as the semi-norm SPCA, and its $\ell_1$-relaxation, corresponding to an arbitrary semi-norm $\phi$:

%

\begin{equation}\label{eq:SPCA} \tag{\textup{Semi-norm SPCA}}
\begin{array}{rc}
\optz \triangleq\textup{max}_x & \phi(x) \\
\textup{s.t.} & \|x\|_2 \leq 1 \\
& \|x\|_0 \leq k,
\end{array} \end{equation}

\begin{equation} \label{eq:SPCA1} \tag{$\ell_1$-norm relaxation}
\begin{array}{rcl}
\opto \triangleq\textup{max}_x & \phi(x)& \\
\textup{s.t.} & \|x\|_2& \leq 1 \\
& \|x\|_1 &\leq \sqrt{k}.
\end{array} 
\end{equation}

In order to convert a solution for the \ref{eq:SPCA1} to a solution for \ref{eq:SPCA}, we  consider the simple randomized rounding procedure of~\cite{FountoulakisKKD17}: 

\begin{algorithm}[H] 
\caption{Randomized rounding of solution of $\ell_1$-relaxation}
\label{algo}
		\begin{algorithmic}[1]
		\State \textbf{Input:} the optimal solution $x$ to the \ref{eq:SPCA1}, and parameters $\gamma \in (0,1)$, $g\in \mathbb{R}_+$
    \smallskip
    	\State Let $p_i = \min\{s\frac{|x_i|}{\|x\|_1} ,1\}$, where $s = \gamma \cdot k$
	    \State Let $\e_i \in \{0,1\}$ take the value 1 with probability $p_i$, and the value 0 with probability $1-p_i$
	    \State Let the $i$-th coordinate of the randomly rounded solution be: $$X_i = \frac{1}{p_i}\,x_i\e_i$$
	    \State Output the solution $\frac{X}{g}$ 
    \end{algorithmic}
  \end{algorithm}

Our main result is an analysis of this procedure that shows that the~\ref{eq:SPCA1} is within a constant factor of the~\ref{eq:SPCA}.

\begin{theorem}\label{thm:main}
For any semi-norm $\phi:\mathbb{R}^n \rightarrow \mathbb{R}_{+}$ and $k \ge 15$, we have that $$\optz \leq \opto\leq 2.95 \cdot \optz.$$ 

Moreover, with positive probability, the solution $\frac{X}{g}$ output by  Algorithm \ref{algo} with $\gamma = 0.4051$ and $g = 2.996$ is feasible for the~\ref{eq:SPCA} problem and satisfies: $\phi(\frac{X}{g}) \ge \frac{1}{3.25}\,\opto$.
\end{theorem}

We note that the constants $2.95$ and $3.25$ can be improved if one considers higher values of the lower bound on $k$. Also with a small additional loss to the constant $3.25$, the success probability of the algorithm can be boosted to an arbitrary constant (by also running the rounding procedure multiple times). 

The high-level idea of the proof is the following: We need to show that with positive probability, $\frac{X}{g}$ is feasible for the~\ref{eq:SPCA} and has large objective value. Standard concentration shows that feasibility holds with ``large'' constant probability. To control the value, notice that the rounding in unbiased, namely $\E X = x$, and that $\phi$ is convex. Thus, the \textbf{expected} objective value of our unscaled solution is large: $\E \phi(X) \ge \phi(\E X) = \phi(x) = \optz$ (the scaling only introduces an additional $\frac{1}{g}$ factor in the bound).

The issue is that, in principle, our solution $X$ could take a very objective large value with very small probability (and this happening when it is infeasible), and taking very small value with probability close to 1. To show that this does not happen, we need to control the upper tail of $\phi(X)$ (and with something more effective than Markov's inequality). 

	However, it is not clear how to obtain concentration for $\phi(X)$ since we cannot control its ``Lipschitzness''; for example, in the special case $\phi = \|.\|_A$, we do not have any assumptions on the magnitude of the entries of $A$, and in particular its relationship to $\optz$. 

	To handle this issue, we use solely $\|X\|_0$ and $\|X\|_2$ to control $\phi(X)$. More specifically, we upper bound the largest possible objective value of a solution with $\|.\|_0 = t$ and $\|.\|_2 = w$, and show that it is at most $\approx w \sqrt{t/k}\, \optz$ (Lemma \ref{lem:rhs}); this provides and upper bound on $\phi(X)$ as long as $\|X\|_0 \le t$ and $\|X\|_2 \le w$. Then, the we obtain the desired control over the behavior of $\phi(X)$ by employing concentration for $\|.\|_0$ and $\|.\|_2$ and carefully integrating over $t$ and $w$.

\medskip

A natural question is how good the constant $2.95$ presented in Theorem~\ref{thm:main} is; we present a lower bound on this constant.
\begin{theorem}\label{thm:lower}
There exists a rank one positive-semidefinite matrix $A$ such that with $\phi = \|.\|_A$ we have that $$\opto \geq 1.32\cdot \optz.$$
\end{theorem}
	
	
	Since there is a big gap between the upper and lower bounds obtained on the worst-case value of the multiplicative constant factor, it is an open question which of them is closer to the actual worst-case bound. In our limited computational experiments, we saw ratios significantly lesser than $1.32$, so we speculate that the lower bound of $1.32$ is perhaps closer to the actual constant.  

\section{Proof of Theorem \ref{thm:main}}
\label{sec:proof}


\subsection{Preliminaries}\label{sec:pre}

In this section we collect a few technical results that will be needed in the sequel. The first is a simple observation on the arithmetico-geometric series, for which we include a proof for completeness.

\begin{lemma}\label{lem:sum}
$\sum_{t = k}^n te^{-t} \leq \eta(k) \triangleq e^{-k}\left[\frac{k e^2 - (k -1) e}{ (e - 1)^2} \right]$.
\end{lemma}
\begin{proof} Let $S: = \sum_{t = k}^n te^{-t}$. Then $eS = \sum_{t = k}^n t e^{-(t -1)}$ and therefore
\begin{eqnarray}
(e - 1)S = ke^{-k +1} + \sum_{t = k +1}^n e^{-(t-1)} - ne^{-n} \leq ke^{-k +1} + \sum_{t = k +1}^{\infty} e^{-(t-1)} \leq ke^{-k +1} + \frac{e^{-k}}{1 - e^{-1}},
\end{eqnarray}
	and therefore $S \leq e^{-k}\left[\frac{k e^2 - (k -1) e}{ (e - 1)^2} \right]$.
\end{proof}
	
	We will also need the following conditional layer-cake decomposition, which follows, for instance, by applying the standard layer-cake decomposition~\cite{lieb} to the law of $Z$ conditioned on $Z \ge t$.
	
\begin{lemma}[Layer-cake Decomposition] \label{lemma:layer}
Let $Z$ be a non-negative random variable. Then for any $t \ge 0$ $$\E[Z \mid Z \ge t] \Pr(Z \ge t) = t \cdot \Pr(Z \ge t) + \int_t^{\infty} \Pr(Z \ge \alpha) d\alpha.$$ 
\end{lemma}
	
Next we present a multiplicative Chernoff (or Poisson-type) bound that has good constants for our regime (notice the constant 1 in front of $t$ in the exponent) and has a simple form that we can later integrate over; the proof is standard and is presented in Appendix \ref{app:chernoff}. 

\begin{lemma} \label{lemma:chernoff}
	Consider independent random variables $Z_1, Z_2, \ldots, Z_n$ where $Z_i \in [0, b_i]$. Letting $\mu_i = \E Z_i$, we have
	\begin{align*}
		\Pr\left(\sum_i Z_i \ge t \right) \le e^{\sum_i \mu_i (1 + (e-2) b_i)} \cdot e^{-t}.
	\end{align*}
\end{lemma}

We will also need the following estimate on Gaussian integrals. 

\begin{lemma}[Lemma 2, Chapter 7 of~\cite{feller}] \label{lemma:gauss}
For  all $x \geq 0$, $$\int_x^\infty e^{-\alpha^2}\,d\alpha \le \frac{e^{-x^2}}{2x}.$$
\end{lemma}


\subsection{Value Function with Respect to Right-hand Side}\label{sec:vfrhs}

We now bound how much $\OPT$ can change as we change the right-hand side of the~\ref{eq:SPCA}. To make this precise, for $t \in \mathbb{Z}_{+}$ and $w \ge 0$ we define  
\begin{eqnarray}
\OPT(t,w) & \triangleq & \textup{max}_{x}\ \phi (x) \notag\\
 		  & \textup{s.t.} & \|x\|_2 \leq w \label{eq:rhs}\\
		  && \|x\|_0 \leq t.  \notag	
\end{eqnarray}
Thus $\OPT(k, 1)$ is the same as $\OPT$. The main result of this section is the following upper bound. 

\begin{lemma}[RHS Changes]\label{lem:rhs}
		Let $t \in \mathbb{Z}_+$ and $w \geq 0$. Then 
		$$\OPT(t,w) \leq \left(w\sqrt{\left\lceil\frac{t}{k} \right\rceil} \right) \OPT.$$
\end{lemma}

	To prove this result, we start with the following observation which controls the dependence on $w$ and follows directly from the positive homogeneity of the functions $\phi$ and $\|x\|_2$.
	
\begin{proposition} \label{prop:wOPT}
For every $w \ge 0$, $\OPT(t,w) = w \cdot \OPT(t,1)$.
\end{proposition}

	The following proposition then controls the dependence on $t$.

\begin{proposition} \label{prop:upper-bound-SPCA-02}
For every $t \geq k$, $\OPT(t,1) \leq \sqrt{\left\lceil\frac{t}{k} \right\rceil} \OPT(k, 1)$.
\end{proposition}

\begin{proof}
This essentially follows from subadditivity of $\phi$. More precisely, let $x^{\ast}$ be an optimal solution corresponding to $\OPT(t,1)$, i.e. optimal for \eqref{eq:rhs} with right-hand side $w = 1$. Since $\|x\|_0 \le t$, consider a decomposition $x^* = x^1 + \ldots + x^{\lceil\frac{t}{k}\rceil}$ where each vector $x^i$ has $\|x^i\|_0 \le k$ and they have disjoint support. By subadditivity of $\phi$ we have 
\begin{align}\label{eqlocal:1}
	\OPT(t,1) = \phi(x^{\ast}) = \phi\left(\sum_{i = 1}^{\lceil\frac{t}{k}\rceil}x^i\right) \leq \sum_{i = 1}^{\lceil\frac{t}{k}\rceil} \phi(x^i).
\end{align}
	But the scaled vector $\frac{x^i}{\|x^i\|_2}$ is a feasible solution to the optimization problem corresponding to $\OPT(k,1)$, and so using the positive homogeneity of $\phi$ we have for each $i$
	\begin{align*}
		\phi(x^i) = \|x^i\|_2 \phi\left(\frac{x^i}{\|x^i\|_2}\right) \le \|x^i\|_2 \OPT(1,k),
	\end{align*}	
and thus 
	\begin{align}
		\OPT(t,1) \le \OPT(1,k)\cdot \sum_{i = 1}^{\lceil \frac{t}{k} \rceil} \|x^i\|_2. \label{eq:proofRHS}
	\end{align}

Moreover, by construction the $x^i$'s are orthogonal to each other, and hence 
$$1 = \|x^{\ast}\|_2^2 = \sum_{i = 1}^{\lceil\frac{t}{k}\rceil}\|x^i\|_2^2.$$ Using the standard $\ell_1$-$\ell_2$ comparison inequality $\sum_{i=1}^d |a_i| \le \sqrt{d} \cdot  \sum_{i=1}^d a^2_i$, we obtain that 
$\sum_{i = 1}^{\lceil\frac{t}{k}\rceil} \|x^i\|_2\leq \sqrt{\lceil\frac{t}{k}\rceil}$. Substituting this in \eqref{eq:proofRHS} then concludes the proof.
\end{proof}

\begin{proof}[Proof of Lemma \ref{lem:rhs}]
Follows directly by combining Propositions \ref{prop:wOPT} and \ref{prop:upper-bound-SPCA-02}:
$$\OPT(t,w) \leq w \cdot \OPT(t,1) \le  \left(w\sqrt{\left\lceil\frac{t}{k} \right\rceil} \right) \OPT(k,1).$$
\end{proof}	


\subsection{Concentration Inequalities for $\ell_0$-norm}\label{sec:l0}

Note that $\|X\|_0 = \sum_{i = 1}^n \varepsilon_i$ is the sum of independent Bernoulli random variables. Moreover, since $\varepsilon_i = 1$ with probability $p_i = \textup{min}\{s\frac{|x_i|}{\|x\|_1},1\}$ and $s= \gamma \cdot k$, we have $\E\|X\|_0 = \sum_{i \in [n]}p_i \leq \gamma k \ll k$; thus $X$ (and hence the scaled version $\frac{X}{g}$) satisfies the sparsity constraint $\|X\|_0 \le k$ in expectation. Moreover, applying Lemma \ref{lemma:chernoff} with $b_i = 1$ and $\mu_i = p_i$ we obtain the following tail bound. 

\begin{lemma} \label{lemma:l0}
$$\Pr\left( \|X\|_0 \geq t \right) \leq e^{c_1 \cdot k - t},$$ where $c_1 = (e-1)\gamma$. 
\end{lemma}

As a consequence, we have the following estimate for the expected value on the tail of $\|X\|_0$.
	\begin{corollary} \label{cor:intl0}
		For all $y \ge 0$, $$\sum_{t \in \Z_{+}, t \ge y} t \,\Pr\left(\|X\|_0 = t\right) \le e^{c_1 k - y} (y + 1).$$
	\end{corollary}

	\begin{proof}
		Since the left-hand side equals $\E[\|X\|_0 \mid \|X\|_0 \ge y] \Pr(\|X\|_0 \ge y)$, employing the Layer-cake Decomposition and the lemma above we have
			\begin{align*}
			\sum_{t \in \Z_{+}, t \ge y} t \,\Pr\left(\|X\|_0 = t\right) &= y \, \text{Pr}(\|X\|_0 \geq y) + \int_{\alpha = y}^{\infty} \text{Pr}(\|X\|_0 \geq \alpha) \,d\alpha\\  
			& \leq  e^{c_1 k} \left(y e^{-y} + \int_{\alpha = y}^{\infty}e^{-\alpha}\,d\alpha\right)\\
			&=  e^{c_1 k - y} (y + 1).
		\end{align*}
	\end{proof}
	

\subsection{Concentration inequalities for $\ell_2$-norm}\label{sec:l2}

	Now we control the $\ell_2$-norm $\|X\|_2$.
It is straightforward to verify that $\E\|X\|_2 \leq \sqrt{\frac{1}{\gamma} +1} \approx 1$; in particular, the scaled solution $\frac{X}{g}$ satisfies the restriction $\|\frac{X}{g}\|_2 \le 1$ in expectation. We use Lemma \ref{lemma:chernoff} to give a simple proof of a dimension-free concentration for $\|X\|_2$ in our setting.\footnote{More general results of this type with worse constants can be obtained, for instance, via the entropy method, see Theorem 6.10 of \cite{lugosi}.}

	\begin{lemma} \label{lemma:l2} We have
	 $$\Pr(\|X\|_2 \ge  t) \le c_2 \cdot e^{-t^2},$$ where $c_2  \triangleq e^{e - 1 + \frac{1}{\gamma} + \frac{(e - 2)}{\gamma^3 k}}$.
	\end{lemma}
	
	\begin{proof}
		Squaring on both sides, equivalently we need to upper bound the probability that $\sum_i X^2_i = \|X\|_2^2 \ge t^2$. Notice that the random variable $X^2_i$ is in the interval $[0, x_i^2/p_i^2]$, and its expectation is $\E X^2_i = \frac{x^2_i}{p_i}$. 
Thus, applying Lemma \ref{lemma:chernoff} to $(X^2_i)_i$ we obtain
		\begin{align}
			\Pr\left(\|X\|_2 \ge t \right) =  \Pr\left(\sum_i X^2_i \ge t^2 \right) \le e^{\sum_i \frac{x^2_i}{p_i} + (e-2) \sum_i \frac{x^4_i}{p^3_i} }. \label{eq:l2}
		\end{align}
		
	Using the fact that $\|x\|_1 \le \sqrt{k}$ and $\|x\|_2 \le 1$, we can upper bound the first sum in the exponent by 
	\begin{align*}
	  \sum_{i = 1}^n \frac{x_i^2}{p_i} = \sum_{i: p_i = \frac{s|x_i|}{\|x\|_1}} \frac{x_i^2}{p_i} + \sum_{i: p_i = 1} \frac{x_i^2}{p_i} \le s \|x\|_1^2 + \|x\|_2^2 \le \frac{k}{s} + 1 = \frac{1}{\gamma} + 1,
	\end{align*}
	where the last inequality uses the definition $s = \gamma \cdot k$. The other summation can be upper bounded similarly as 
	\begin{align*}
	  & \sum_{i = 1}^n \frac{x_i^4}{p_i^3} = \sum_{i: p_i = \frac{s|x_i|}{\|x\|_1}} \frac{x_i^4}{p_i^3} + \sum_{i: p_i = 1} \frac{x_i^4}{p_i^3} \leq \frac{1}{\gamma^3 k} + 1.	\end{align*} 
	 Plugging these bounds on inequality \eqref{eq:l2} concludes the proof. 
\end{proof}	

As a consequence, we have the following estimate for the expected value on the tail of $\|X\|_2$.

\begin{corollary} \label{cor:intl2}
	For any $t \ge 0$, we have $$\sum_{w \ge t} w \,\Pr\left(\|X\|_2 = w\right) \le c_2 \left(t + \frac{1}{2 t}\right) \,e^{-t^2}.$$ 
\end{corollary}
	
\begin{proof}
	Employing the Layer Cake Lemma and Lemma \ref{lemma:l2} above we have
		\begin{align*}
		\sum_{w \geq t} w \, \text{Pr}(\|X\|_2 = w) & = \E\left[\|X\|_2 \mid \|X\|_2 \geq t\right] \cdot \Pr\left(\|X\|_2 \geq t\right) \\
			& = t \cdot \Pr\left(\|X\|_2 \geq t\right) + \int_{t} ^{\infty} \Pr\left(\|X\|_2 \geq \alpha\right) ~d\alpha \\
			& \leq c_2 \left(t\,e^{-t^2} + \int_{t} ^{\infty} e^{- \alpha^2} d\alpha \right)\\
		& \le  c_2 \frac{2t^2 + 1}{2 t}\,e^{-t^2} &\textrm{(Lemma \ref{lemma:gauss})},
		\end{align*}
		which concludes the proof of the corollary.
\end{proof}
	

\subsection{Controlling the Objective Value}\label{sec:obj}

As mentioned in the introduction, since $\phi$ is convex, Jensen's inequality gives $\E (\phi(X)) \ge \phi(\E X) = \phi(x) = \opto$, which is at least $\OPT$ (thus, by positive homogeneity $\E \phi(\frac{X}{g}) \ge \frac{\OPT}{g}$). We break up this expectation in the cases where the scaled solution $\frac{X}{g}$ is feasible or not for the~\ref{eq:SPCA}:
	\begin{align}
		\opto \leq \E (\phi(X)) =&~ \E \bigg[\phi(X) ~\bigg|~ \|X\|_0 \le k, \|X\|_2 \le g \bigg]\Pr(\|X\|_0 \le k, \|X\|_2 \le g ) \nonumber \\
		&+ \E \bigg[\phi(X) ~\bigg|~ \|X\|_0 \ge k + 1 \textrm{ or } \|X\|_2 > g \bigg] \Pr(\|X\|_0 \ge k + 1 \textrm{ or } \|X\|_2 > g). \label{eq:wrap}
	\end{align}

	%
	
	In the next lemma we upper bound the contribution of the second term in the right-hand side, i.e., the contribution to the value by infeasible scenarios.  
	
	
	
\begin{lemma}\label{lemma:truncation} If $k \ge 10$ and $g > 1$, we have 
	\begin{align*}
 \E \bigg[\phi(X) ~\bigg|~ \|X\|_0 \ge k + 1 \textrm{ or } \|X\|_2 > g \bigg] \cdot \Pr\bigg(\|X\|_0 \ge k + 1 \textrm{ or } \|X\|_2 > g\bigg) \le \alpha \OPT,
	\end{align*}
		 where $\alpha = \frac{3}{2 \sqrt{k}} \cdot c_2 \eta(k + 1) + \frac{\sqrt{2}}{\sqrt{k}} \cdot e^{-(1-c_1) k - 1} (k + 2)  + c_2 \cdot (g + \frac{1}{2g}) e^{- g^2}$.
	\end{lemma}
	
	\begin{proof}
We first simplify the notation and define 
	\begin{align*}
		f(t,w) &= \E\big[\phi(X) \mid \|X\|_0 = t, \|X\|_2 = w\big]\\
		p(t,w) &= \Pr\big(\|X\|_0 = t, \|X\|_2 = w\big)\\
		p(t) &= \Pr\big(\|X\|_0 = t\big).
	\end{align*} 
	Thus, we can write the left-hand side of the lemma as 	
	\begin{align*}
		\E \big[\phi(X)~\big|~ \|X\|_0 \ge k + 1 \textrm{ or } \|X\|_2 > g \big] \cdot \Pr\big(\|X\|_0 \ge k + 1 \textrm{ or } \|X\|_2 > g \big) \\
		= \sum_{(t,w)\, :\, t \ge k + 1 \textrm{ or } w > g} f(t,w)\, p(t,w). 
	\end{align*}
	Since $X$ only takes finitely many different values, notice that the sum in the right-hand side has finitely many non-zero terms. To control this sum, we are going to use Lemma \ref{lem:rhs} to upper bound $f(t,w)$, and concentration of $\|X\|_0$ and $\|X\|_2$ (Lemmas \ref{lemma:l0} and \ref{lemma:l2} respectively) to upper bound $p(t,w)$. However, concentration of $\|X\|_0$ is only helpful to control the terms with large $t$, and concentration of $\|X\|_2$ to control the terms with large $w$. To be able to effectively cover all terms, we need a careful partition of the sum (see Figure \ref{fig:1}):
	\begin{align}
		&\sum_{(t,w)\, :\, t \ge k + 1 \textrm{ or } w > g} f(t,w)\, p(t,w) \leq \sum_{(t,w)\, :\, t \ge k + 1 \textrm{ or } w \geq g} f(t,w)\, p(t,w) \notag\\
		&~~~~~~~~~~~~~\le  \underbrace{\sum_{t \geq k + 1} \sum_{w \geq \sqrt{t}} f(t,w)\,p(t,w)}_{\text{Sum A.1}} + \underbrace{\sum_{t \geq k  + 1} \sum_{w \leq \sqrt{t}} f(t,w) \cdot p(t,w)}_{\text{Sum A.2}} + \underbrace{\sum_{t \le k} \sum_{w \geq g} f(t,w)\,p(t,w)}_{\text{Sum B}} \label{eq:big}
	\end{align}	
	
	\begin{figure}[h]
	\begin{center}
		\includegraphics[width=0.5\textwidth]{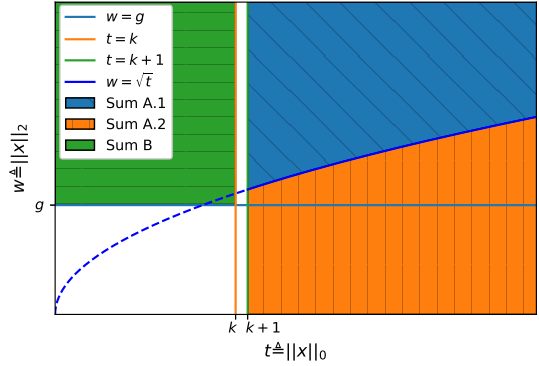}
		\end{center}
		\caption{Visual representation of the various sums}
		\label{fig:1}
	\end{figure}
	
	We upper bound each of these sums separately. 
		

	\paragraph{Sum A.1:} We upper bound this term by $\lesssim \eta(k) \OPT  \lesssim e^{-k} \OPT$.
	
	From Lemma \ref{lem:rhs} we have that for $t \ge k$ 
	\begin{align}
		f(t,w) \le w \sqrt{ \left\lceil \frac{t}{k} \right\rceil}\, \OPT \le w \sqrt{ \frac{2t}{k}}\, \OPT, \label{eq:sumA10}
	\end{align}
	and also $p(t,w) \le \Pr(\|X\|_2 = w)$. Thus, fixing $t$ and adding over $w \ge \sqrt{t}$ we get
		\begin{align}
			\sum_{w \geq \sqrt{t}} f(t,w)\,p(t,w) & \leq \OPT \,\sqrt{\frac{2t}{k}}\,\left(\sum_{w \geq \sqrt{t}} w \, \text{Pr}(\|X\|_2 = w)\right). \label{eq:sumA11}
		\end{align}
	Using Corollary \ref{cor:intl2} and the fact $t \ge k + 1 \ge 11$, the sum inside the bracket on the right-hand side of (\ref{eq:sumA11}) is at most $c_2 1.05 \sqrt{t}\,e^{-t}$.
	Employing this bound on inequality \eqref{eq:sumA11} and adding over all $t \ge k + 1$ we obtain 
		\begin{align*}
			\textbf{Sum A.1} & \leq \left( c_2 \frac{\sqrt{2}\cdot  1.05 }{\sqrt{k}} \cdot \sum_{t \geq k + 1} t e^{- t} \right) \OPT \leq c_2 \frac{3}{2 \sqrt{k}} \eta(k + 1) \OPT,
		\end{align*}
		where the final inequality follows from Lemma \ref{lem:sum}.
		

	\paragraph{Sum A.2:} For $w \le \sqrt{t}$, and using Lemma \ref{lem:rhs} we obtain $$f(t,w) \le  t \sqrt{ \frac{2}{k}}\, \OPT,$$ and thus the sum A.2 can be upper bounded
		\begin{align*}
		\sum_{t \geq k + 1} \sum_{w \le \sqrt{t}}	f(t,w) \, p(t,w) &\le \sqrt{\frac{2}{k}} \OPT \sum_{t \geq k + 1} t \sum_{w \le \sqrt{t}} p(t,w) \notag\\
			& \le \sqrt{\frac{2}{k}} \OPT \sum_{t \geq k + 1} t  \Pr\left(\|X\|_0 = t\right) \\
			& \le \sqrt{\frac{2}{k}} \OPT e^{c_1 k - k - 1} (k + 2) & \textrm{(Corollary \ref{cor:intl0})}.
		\end{align*}
		

	\paragraph{Sum B:} For $t \le k$, Lemma \ref{lem:rhs} gives that $f(t,w) \le w \OPT$, and thus sum B can be upper bounded as 
		\begin{align*}
			 \sum_{t \leq k} \sum_{w \geq g} f(t,w)\, p(t,w) &\leq  \sum_{t \leq k} \sum_{w \geq g} w \OPT \cdot p(t,w) \nonumber \\
			& \leq  \OPT  \sum_{w \geq g} w \Pr\left(\|X\|_2 = w\right) \\
            & \leq c_2 \left(g + \frac{1}{2g}\right) e^{-g^2} \OPT & (\textrm{Corollary \ref{cor:intl2}}). 
		\end{align*}
	\medskip
	Employing these bounds on inequality \eqref{eq:big} concludes the proof of the lemma. 
	\end{proof}


	\subsection{Conclusion of the Proof of Theorem \ref{thm:main}}
 Taking a union bound, the probability that the $\frac{X}{g}$ is feasible is at least $$1 - \Pr(\|X\|_0 \geq k + 1) - \Pr(\|X\|_2 \ge g).$$ One can verify that with the setting $\gamma = 0.44$ and $g = 2.69$, Lemma \ref{lemma:l0} and \ref{lemma:l2} imply that this quantity is strictly positive. 

	Moreover, combining equation \eqref{eq:wrap} and Lemma \ref{lemma:truncation}, and using the fact that $\frac{X}{g}$ is feasible with non-zero probability, we have:

	\begin{gather*}
		\E \bigg[\phi\bigg(\frac{X}{g}\bigg) ~\bigg|~ \bigg\|\frac{X}{g}\bigg\|_0 \leq k, \bigg\|\frac{X}{g}\bigg\|_2 \leq 1 \bigg]\Pr\left(\bigg\|\frac{X}{g}\bigg\|_0 \leq k, \bigg\|\frac{X}{g}\bigg\|_2 \leq 1 \right)  \ge \frac{\opto - \alpha \OPT}{g}\\
		\Rightarrow \E \bigg[\phi\bigg (\frac{X}{g}\bigg) ~\bigg|~ \bigg\|\frac{X}{g}\bigg\|_0 \leq k, \bigg\|\frac{X}{g}\bigg\|_2 \leq 1 \bigg] \ge \frac{\opto - \alpha \OPT}{g}.
	\end{gather*}
	Therefore, there \textbf{exists} a scenario among the ones where $\frac{X}{g}$ is feasible where $\phi(\frac{X}{g}) \ge \frac{\opto - \alpha \OPT}{g}$. Since $\OPT \ge \phi(\frac{X}{g}) \ge \frac{\opto - \alpha \OPT}{g}$ implies $(g + \alpha) \OPT \ge \opto$. Verifying that with our setting of $g = 2.69, \gamma = 0.44, k = 15$, we have $\OPT \geq \frac{1}{2.95} \opto$ concludes the proof of the first part of the theorem. 
	
To prove the second part of the theorem, similar to the above, if the probability that the $\frac{X}{g}$ is feasible is positive, then we have that $\E[\phi(\frac{X}{g})| \frac{X}{g} \textup{ is feasible}] \ge \frac{\opto - \alpha \OPT}{g}$. Thus if the probability that the $\frac{X}{g}$ is feasible is positive, we obtain that with positive probability, $\frac{X}{g}$ is both feasible and satisfies $\phi(\frac{X}{g}) \geq \frac{1 - \alpha}{g} \opto$ (the last inequality follows from $\optz \leq \opto$). Setting $g = 2.996$, $\gamma = 0.4051$ for $k = 15$, we have $\phi(\frac{X}{g}) \ge \frac{1}{3.25} \opto$ which concludes the proof of the second part of the theorem. 

\section{Proof of Theorem~\ref{thm:lower}}\label{sec:lower}
We begin with a simple observation.  

\begin{observation}
Suppose $A = (x^*)(x^*)^T$ where $x^* \in \mathbb{R}^{n}_{+}$, $\|x^*\|_2 = 1$, and $\|x^*\|_1 \leq \sqrt{k}$, and consider the problems, \ref{eq:SPCA} and \ref{eq:SPCA1} with objective function $\phi = \|.\|_A$. Then we have $\opto = 1$.  Moreover, if the coordinates of $x^*$ are sorted in non-increasing order, then $\optz = \sum_{i = 1}^k x_i^2$.
\end{observation}

Therefore, in order to find instances where the ratio $\frac{\opto}{\optz}$ is large, we can solve the following optimization problem:
\begin{eqnarray}
\begin{array}{rcl}
\textup{min}_x & \sum_{i = 1}^k x_i^2 & \\
\textup{s.t.} & \|x\|_2& = 1 \\
& \sum_{i = 1}^n x_i & \leq \sqrt{k} \\
& - x_i + x_{i + 1} & \leq 0, ~ i = 1, \ldots, n - 1 \\
& - x_n & \leq 0,
\end{array} \label{eq:bound-SPCA}
\end{eqnarray}
We show that this optimization problem can be reduced to a four variable optimization problem. In order to do so, note that the above problem is equivalent to the following problem:
\begin{eqnarray}
\begin{array}{rcl}
\textup{min}_{x,a,G,H,C,D} & G & \\
\textup{s.t.} & \sum_{i = 1}^k x_i^2 & = G \\
& \sum_{i = k + 1}^n x_i^2 & = H \\
& G + H & = 1 \\
& \sum_{i = 1}^k x_i & = C \\
& \sum_{i = k + 1}^n x_i & = D \\
& C + D & \leq \sqrt{k} \\
& x_1, \ldots, x_k & \geq a \\
& x_{k + 1}, \ldots, x_n & \leq a \\
& x_{k + 1}, \ldots, x_n & \geq 0.
\end{array} \label{eq:bound-SPCA1}
\end{eqnarray}

In order to solve this problem we first determine some bounds on the new variables $a, G, H, C, D$.

\begin{proposition}\label{prop:bds} 
	Let $x, a, G, H, C, D$ be a feasible solution for \eqref{eq:bound-SPCA1}. Then:
\begin{enumerate}
\item $a \geq 0$
\item $ka \le C \le \sqrt{k}$
\item $0 \le D \le \sqrt{k} - C$
\item $\underbrace{\frac{D^2}{n - k}}_{H_{\text{Lower}}} \le H \le \underbrace{\left\lfloor \frac{D}{a} \right\rfloor a^2 + \left( D - \left\lfloor \frac{D}{a} \right\rfloor a \right)^2 }_{H_{\text{Upper}}}$, assuming $n - k\geq \lfloor \frac{D}{a} \rfloor + 1$
\item $\underbrace{\frac{C^2}{k}}_{G_{\text{Lower}}} \le G \le \underbrace{(C - (k - 1)a)^2 + (k - 1) a^2 }_{G_{\text{Upper}}}$
\end{enumerate}
\end{proposition}

\begin{proof} Items 1 through 3 follow directly from the constraints in \eqref{eq:bound-SPCA1}.


	\smallskip
	\noindent \emph{Item 4.} The upper bound comes from maximizing $\sum_{i = k +1}^n x_i^2$ subject to the condition $\sum_{i = k +1}^n x_i = D, \ x_i \in [0, \ a]\textup{ for all }  i \in \{k +1, \dots, n\}$ (assuming $n - k\geq \lfloor \frac{D}{a} \rfloor + 1$). The lower bound on $H$ is obtained by minimizing $\sum_{i = k +1}^n x_i^2$ subject to the condition $\sum_{i = k +1}^n x_i = D$. Note that the optimal solution is setting $x_i = \frac{D}{n -k}$ for all $x_i$, and under the assumption of $n - k\geq \lfloor \frac{D}{a} \rfloor + 1$ each of these $x_i$'s is less than of equal to $a$.

	\smallskip
	\noindent \emph{Item 5.}  The upper bound comes from maximizing $\sum_{i = 1}^k x_i^2$ subject to the condition $\sum_{i = 1}^k x_i = C, \ x_i \geq a \textup{ for all }i \in [k]$. The lower bound on $G$ is obtained by minimizing $\sum_{i = 1}^k x_i^2$ subject to the condition $\sum_{i = k +1}^n x_i = C$.
\end{proof}

\begin{proposition}\label{prop:fndsol}
Suppose there exists $a$, $C$, $D$ satisfying (1), (2), (3) of Proposition \ref{prop:bds} such that $G_{\text{Upper}} + H_{\text{Upper}}\geq 1$. Let $G^* = \textup{max} \{ G_{\text{Lower}}, 1 - H_{\text{Upper}}\}$. Then exists a vector $x \in \mathbb{R}^n_{+}$ satisfying the feasible region of (\ref{eq:bound-SPCA1}) with objective function value equal to $G^*$.
\end{proposition}
\begin{proof}
Note that since $G_{\text{Upper}} + H_{\text{Upper}}\geq 1$, $G^*= \textup{max} \{ G_{\text{Lower}}, 1 - H_{\text{Upper}}\}$ is the smallest value in the interval $[G_{\text{Lower}}, \ G_{\text{Upper}}]$ such that there exists $H^{*} \in [H_{\text{Lower}}, \ H_{\text{Upper}}]$ satisfying $G^* + H^* = 1$.

Via the proof of Proposition~\ref{prop:bds}, there exists a solution $x_{\text{Upper}} \in \mathbb{R}^{n - k}_{+}$ satisfying, $\| x_{\text{Upper}} \|_2 = H_{\text{Upper}}$, $\| x_{\text{Upper}} \|_1 = D$, and $ (x_{\text{Upper}})_i \leq a \ \forall i \in [n - k]$. Similarly, there exists a solution $x_{\text{Lower}} \in \mathbb{R}^{n - k}_{+}$ satisfying, $\| x_{\text{Lower}} \|_2 = H_{\text{Lower}}$, $\| x_{\text{Lower}} \|_1 = D$, and $ (x_{\text{Lower}})_i \leq a \ \forall i \in [n - k]$. Since $\|\cdot\|_2$ is a continuous function there is a convex combination of $x_{\text{Upper}}$ and $x_{\text{Lower}}$, say $y \in \mathbb{R}^{n - k}_{+}$ satisfying $\| y\|_2 = H^*$, $\| y \|_1 = D$, and $ (y)_i \leq a \ \forall i \in [n - k]$.

Now using the same argument for $G$, we can obtain $z \in \mathbb{R}^k_{+}$ such that $\| z\|_2 = G^*$, $\| z \|_1 = C$, and $ (z)_i \geq a \ \forall i \in [n - k]$. Thus, the augmented vector, $(z^{\top} \ y^{\top})^{\top}$ satisfies the feasible region of (\ref{eq:bound-SPCA1}) with objective function value equal to $G^*$.
\end{proof}

As a consequence of Proposition~\ref{prop:fndsol}, the optimization problem (\ref{eq:bound-SPCA}) may be solved by solving the following problem:
\begin{eqnarray}\label{fprob}
\begin{array}{rl}
\textup{min}_{\theta, a, C, D}& \theta \\
\textup{s.t.} & \frac{1}{\sqrt{k}} \geq a \geq 0\\
& \sqrt{k} \geq C \geq ka \\
& \sqrt{k} - C \geq D \geq 0\\
& \left \lfloor \frac{D}{a} \right\rfloor a^2 + \left( D - \left\lfloor \frac{D}{a} \right\rfloor a \right)^2 + (C - (k - 1)a)^2 + (k - 1) a^2 \geq 1\\
&\theta \geq \frac{C^2}{k} \\
&\theta \geq 1 - \left( \left \lfloor \frac{D}{a} \right\rfloor a^2 + \left( D - \left\lfloor \frac{D}{a} \right\rfloor a \right)^2 \right)
\end{array}
\end{eqnarray}

Note in the above problem, we can always set $D = \sqrt{k} - C$. We solved the above problem numerically (obtaining an upper bound to \eqref{eq:bound-SPCA1}), by just discretizing in the space of $a$ and $C$ variables and taking the best feasible point. {The result of our numerical experiments is presented in Figure~\ref{fig:2}, where the $y$-axis is the reciprocal of the optimal objective function value of problem (\ref{fprob}), which is $\opto/\optz$. Notice that $\opto/\optz$ is increasing with increasing values of $k$, but it seems to converge to a value slightly greater than $1.32$. It can be verified that $k = 10000$, $a = 0.005$, $C = 51$, $D = 49$, $\theta = 0.755$ is a feasible solution for (\ref{fprob}), i.e. $\opto/\optz \geq 1.324$.  This completes the proof of Theorem~\ref{thm:lower}.}

	\begin{figure}[h]
	\begin{center}
		\includegraphics[width=0.75\textwidth]{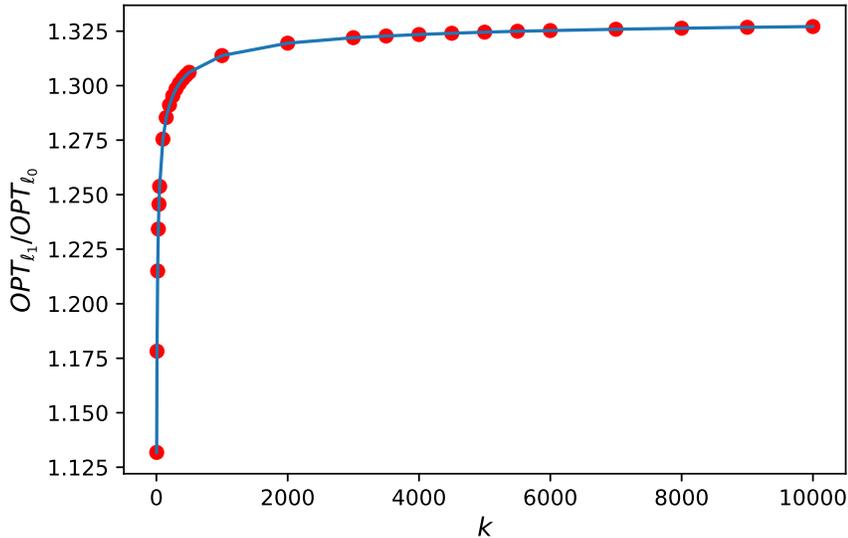}
		\end{center}
		\caption{Result of Problem (\ref{fprob}) for varying values of $k$}
		\label{fig:2}
	\end{figure}

\paragraph{Acknowledgements.} 
Santanu S. Dey would like to acknowledge the support of NSF CMMI grant 1562578.
\bibliographystyle{plain}
\bibliography{ref,rahul_dbm3}   

\begin{thebibliography}{10}

\bibitem{anderson2003}
T.~W. Anderson.
\newblock {\em An Introduction to Multivariate Statistical Analysis}.
\newblock Wiley, New York, 3rd edition, 2003.

\bibitem{birnbaum2013minimax}
Aharon Birnbaum, Iain~M Johnstone, Boaz Nadler, and Debashis Paul.
\newblock Minimax bounds for sparse pca with noisy high-dimensional data.
\newblock {\em Annals of statistics}, 41(3):1055, 2013.

\bibitem{lugosi}
St{\'e}phane Boucheron, G{\'a}bor Lugosi, and Pascal Massart.
\newblock {\em Concentration inequalities: A nonasymptotic theory of
  independence}.
\newblock Oxford university press, 2013.

\bibitem{buhlmann2011statistics}
Peter B{\"u}hlmann and Sara {van-de-Geer}.
\newblock {\em Statistics for high-dimensional data}.
\newblock Springer, 2011.

\bibitem{cadima1995loading}
Jorge Cadima and Ian~T Jolliffe.
\newblock Loading and correlations in the interpretation of principle
  compenents.
\newblock {\em Journal of Applied Statistics}, 22(2):203--214, 1995.

\bibitem{candes2006stable}
Emmanuel~J Candes, Justin~K Romberg, and Terence Tao.
\newblock Stable signal recovery from incomplete and inaccurate measurements.
\newblock {\em Communications on pure and applied mathematics},
  59(8):1207--1223, 2006.

\bibitem{candes2005decoding}
Emmanuel~J Candes and Terence Tao.
\newblock Decoding by linear programming.
\newblock {\em IEEE transactions on information theory}, 51(12):4203--4215,
  2005.

\bibitem{jordan_07}
A.~d'Aspremont, L.~El. Ghaoui, M.~I. Jordan, and G.~R.~G. Lanckriet.
\newblock A direct formulation for sparse pca using semidefinite programming.
\newblock {\em SIAM Review}, 49:434--448, 2007.

\bibitem{donoho2006}
D.~Donoho.
\newblock For most large underdetermined systems of equations, the minimal
  $\ell^1$-norm solution is the sparsest solution.
\newblock {\em Communications on Pure and Applied Mathematics}, 59:797--829,
  2006.

\bibitem{feller}
Willliam Feller.
\newblock {\em An introduction to probability theory and its applications},
  volume~2.
\newblock John Wiley \& Sons, 2008.

\bibitem{FountoulakisKKD17}
Kimon Fountoulakis, Abhisek Kundu, Eugenia{-}Maria Kontopoulou, and Petros
  Drineas.
\newblock A randomized rounding algorithm for sparse {PCA}.
\newblock {\em {TKDD}}, 11(3):38:1--38:26, 2017.

\bibitem{GVL83}
G.~Golub and C.~{Van Loan}.
\newblock {\em Matrix Computations}.
\newblock Johns Hopkins University Press, Baltimore., 1983.

\bibitem{hancock1996face}
Peter~JB Hancock, A~Mike Burton, and Vicki Bruce.
\newblock Face processing: Human perception and principal components analysis.
\newblock {\em Memory \& Cognition}, 24(1):26--40, 1996.

\bibitem{hastie2000gene}
Trevor Hastie, Robert Tibshirani, Michael~B Eisen, Ash Alizadeh, Ronald Levy,
  Louis Staudt, Wing~C Chan, David Botstein, and Patrick Brown.
\newblock 'gene shaving'as a method for identifying distinct sets of genes with
  similar expression patterns.
\newblock {\em Genome biology}, 1(2):research0003--1, 2000.

\bibitem{FHT-09-new}
Trevor Hastie, Robert Tibshirani, and Jerome Friedman.
\newblock {\em The Elements of Statistical Learning, Second Edition: Data
  Mining, Inference, and Prediction}.
\newblock Springer New York, 2 edition, 2009.

\bibitem{hastie2015statistical}
Trevor Hastie, Robert Tibshirani, and Martin Wainwright.
\newblock {\em Statistical learning with sparsity}.
\newblock CRC press, 2015.

\bibitem{hotelling1933analysis}
Harold Hotelling.
\newblock Analysis of a complex of statistical variables into principal
  components.
\newblock {\em Journal of educational psychology}, 24(6):417, 1933.

\bibitem{jenatton2010structured}
Rodolphe Jenatton, Guillaume Obozinski, and Francis Bach.
\newblock Structured sparse principal component analysis.
\newblock In {\em Proceedings of the Thirteenth International Conference on
  Artificial Intelligence and Statistics}, pages 366--373, 2010.

\bibitem{johnstone2009consistency}
Iain~M Johnstone and Arthur~Yu Lu.
\newblock On consistency and sparsity for principal components analysis in high
  dimensions.
\newblock {\em Journal of the American Statistical Association},
  104(486):682--693, 2009.

\bibitem{jolliffe2002principal}
Ian~T Jolliffe.
\newblock Principal component analysis and factor analysis.
\newblock {\em Principal component analysis}, pages 150--166, 2002.

\bibitem{jolliffe2003modified}
Ian~T Jolliffe, Nickolay~T Trendafilov, and Mudassir Uddin.
\newblock A modified principal component technique based on the lasso.
\newblock {\em Journal of computational and Graphical Statistics},
  12(3):531--547, 2003.

\bibitem{journee2010generalized}
Michel Journ{\'e}e, Yurii Nesterov, Peter Richt{\'a}rik, and Rodolphe
  Sepulchre.
\newblock Generalized power method for sparse principal component analysis.
\newblock {\em Journal of Machine Learning Research}, 11(Feb):517--553, 2010.

\bibitem{lieb}
Elliott~H Lieb and Michael Loss.
\newblock Analysis, volume 14 of graduate studies in mathematics.
\newblock {\em American Mathematical Society, Providence, RI,}, 4, 2001.

\bibitem{luss2013conditional}
Ronny Luss and Marc Teboulle.
\newblock Conditional gradient algorithmsfor rank-one matrix approximations
  with a sparsity constraint.
\newblock {\em SIAM Review}, 55(1):65--98, 2013.

\bibitem{mccoy2011two}
Michael McCoy, Joel~A Tropp, et~al.
\newblock Two proposals for robust pca using semidefinite programming.
\newblock {\em Electronic Journal of Statistics}, 5:1123--1160, 2011.

\bibitem{misra2002interactive}
Jatin Misra, William Schmitt, Daehee Hwang, Li-Li Hsiao, Steve Gullans, George
  Stephanopoulos, and Gregory Stephanopoulos.
\newblock Interactive exploration of microarray gene expression patterns in a
  reduced dimensional space.
\newblock {\em Genome research}, 12(7):1112--1120, 2002.

\bibitem{naikal2011informative}
Nikhil Naikal, Allen~Y Yang, and S~Shankar Sastry.
\newblock Informative feature selection for object recognition via sparse
  {PCA}.
\newblock In {\em Computer Vision (ICCV), 2011 IEEE International Conference
  on}, pages 818--825. IEEE, 2011.

\bibitem{paul2012augmented}
Debashis Paul and Iain~M Johnstone.
\newblock Augmented sparse principal component analysis for high dimensional
  data.
\newblock {\em arXiv preprint arXiv:1202.1242}, 2012.

\bibitem{shen2008sparse}
Haipeng Shen and Jianhua~Z Huang.
\newblock Sparse principal component analysis via regularized low rank matrix
  approximation.
\newblock {\em Journal of multivariate analysis}, 99(6):1015--1034, 2008.

\bibitem{Ti96}
R.~Tibshirani.
\newblock Regression shrinkage and selection via the lasso.
\newblock {\em Journal of the Royal Statistical Society, Series B},
  58:267--288, 1996.

\bibitem{uuguz2011two}
Harun U{\u{g}}uz.
\newblock A two-stage feature selection method for text categorization by using
  information gain, principal component analysis and genetic algorithm.
\newblock {\em Knowledge-Based Systems}, 24(7):1024--1032, 2011.

\bibitem{vu2012minimax}
Vincent~Q Vu and Jing Lei.
\newblock Minimax rates of estimation for sparse pca in high dimensions.
\newblock In {\em International Conference on Artificial Intelligence and
  Statistics}, pages 1278--1286, 2012.

\bibitem{wht_09}
DM. Witten, R.~Tibshirani, and T.~Hastie.
\newblock A penalized matrix decomposition, with applications to sparse
  principal components and canonical correlation analysis.
\newblock {\em Biostatistics}, 10(3):515--534, 2009.

\bibitem{zhang2012sparse}
Youwei Zhang, Alexandre d’Aspremont, and Laurent El~Ghaoui.
\newblock Sparse {PCA}: Convex relaxations, algorithms and applications.
\newblock In {\em Handbook on Semidefinite, Conic and Polynomial Optimization},
  pages 915--940. Springer, 2012.

\bibitem{ZY2006}
P.~Zhao and B.~Yu.
\newblock On model selection consistency of lasso.
\newblock {\em Journal of Machine Learning Research}, 7:2541--2563, 2006.

\bibitem{zou2006sparse}
Hui Zou, Trevor Hastie, and Robert Tibshirani.
\newblock Sparse principal component analysis.
\newblock {\em Journal of computational and graphical statistics},
  15(2):265--286, 2006.

\end{thebibliography}

\appendix

\noindent {\LARGE \textbf{Appendix}}
\bigskip

\section{Proof of Lemma \ref{lemma:chernoff}} \label{app:chernoff}

	Using Markov's inequality and independence we have 
	\begin{align}
\text{Pr}\left( \sum_i X_i \geq t \right) & = \text{Pr} \left( e^{\sum_i X_i} \ge e^t \right ) \le \frac{\E e^{\sum_i X_i}}{e^t} = \frac{\prod_i \E e^{X_i}}{e^t}. \label{eq:chernoff}
\end{align} 

But for $x \in [0,b_i]$ we have $e^x \le 1 + x \frac{e^{b_i} - 1}{b_i}$; furthermore, $e^y \le 1 + y + (e-2) y^2$ for $y \in [0,1]$, so employing this to bound $e^{b_i}$ in the previous inequality we obtain $e^x \le 1 + x (1 + (e-2) b_i)$. Therefore, $$\E e^{X_i} \le 1 + \mu_i (1 + (e-2) b_i) \le e^{\mu_i (1 + (e-2) b_i)},$$ where the last inequality follows from $1 + x \le e^x$ that holds for all $x$. Employing this bound on \eqref{eq:chernoff} concludes the proof of the lemma.

\end{document}